\documentclass{amsart}
\usepackage{amsfonts}
\usepackage{amssymb}
\usepackage{amsmath}
\usepackage{tikz}
\usepackage{dsfont}
\usepackage{graphicx}
\usepackage{ mathrsfs }
\newtheorem{mtheorem}{Theorem}
\newtheorem{theorem}{\textbf{Theorem}}[section]

\newtheorem{claim}[theorem]{\textbf{Claim}}

\newtheorem{corollary}{\textbf{Corollary}}

\newtheorem{definition}[theorem]{\textbf{Definition}}
\newtheorem{example}{\textbf{Example}}
\newtheorem{lemma}[theorem]{\textbf{Lemma}}

\newtheorem{remark}[theorem]{\textbf{Remark}}

\numberwithin{equation}{section}

\newcommand{\eqdef}{\stackrel{\scriptscriptstyle\rm def}{=}}

\begin{document}

\title[Synchronization on average]{Random products of maps synchronizing on average}


\author[E. Matias]{Edgar Matias}
\address{Instituto de Matem\'atica, Universidade Federal do Rio de Janeiro, Av. Athos da Silveira Ramos 149, Cidade Universit\'aria - Ilha do Fund\~ao, Rio de Janeiro 21945-909,  Brazil}
\email{edgarmatias9271@gmail.com}

\author[I. Melo]{\'Italo Melo}
\address{Universidade Federal do Piau\' i, Departamento de Matem\' atica, 64049-550,
Ininga - Teresina - PI, Brazil}
\email{italodowell@ufpi.edu.br}

\begin{abstract}
 We present a necessary and sufficient condition for a random product of maps 
on a compact metric space to be (strongly) synchronizing on average.
\end{abstract}


\begin{thanks}{E.M is currently supported by PNPD CAPES.
This work was partially completed while E.M
was supported by PCI award 312487/2016-9 at IMPA and I.M was supported by a CNPq-Brazil doctoral fellowship at IMPA.
The authors warmly thank K. Gelfert and Lorenzo D\' iaz for their useful comments on this paper.  
}\end{thanks}

\maketitle
\section{Introduction}
Let $(\Omega,\mathscr{F},\mathbb{P},\theta)$ be a measure-preserving dynamical system and  $(X,d)$ a compact metric space
 (endowed with the Borel $\sigma$-algebra $\mathscr{B}$). Let $f\colon \Omega\times X\to X$ be a  
 $\mathscr{F}\otimes \mathscr{B}/\mathscr{B}$-measurable map.
 A \emph{random product of maps} on $X$  over the measure-preserving dynamical system
 $(\Omega,\mathscr{F},\mathbb{P},\theta)$ 
 is a map $\varphi\colon \mathbb{N}\times \Omega\times X\to X$ 
defined by
$$
\varphi(n,\omega,x)\eqdef f_{\theta^{n-1}(\omega)}\circ\dots \circ f_{\omega}(x)\eqdef f_{\omega}^{n}(x)\quad \text{for}\quad n\geq 1,\
$$
 where $f_{\omega}\colon X\to X$ is a family of maps given by $f_{\omega}(x)\eqdef f(\omega,x)$.  We often refer to $\varphi$ as a random product of maps $f_{\theta^{i}(\omega)}$ on $X$ over the measure-preserving dynamical system
 $(\Omega,\mathscr{F},\mathbb{P},\theta)$.

In this paper we present a necessary and sufficient condition to
obtain synchronization of the random orbits $f_{\omega}^{n}(x)$, that is, random orbits of different initial points converging to each other. 
   The synchronization effect was first observed  by 
 Huygens \cite{Huygens} in the movement of two pendulum clocks hanging from a wall and since then has been investigated in several areas, see \cite{syncrhonizationbook}. Despite the vast literature on this subject, most of the 
 results are experimental. The first rigorous results about synchronization of independent and identically distributed (i.i.d.) random maps goes back to \cite{Antonov, Fur}, and more recently the synchronization on the circle has been intensively studied, see \cite{Katrin,Goro,KlepAntonov,Malicet}.
 

%
%
 
 For random products there are several ways to introduce the notion of synchronization. In this 
 paper we consider convergence of orbits on average. First we recall that a random product $\varphi$ is \emph{synchronizing 
 on average} according to \cite{Goro} if for every $x, y\in X$, for $\mathbb{P}$-almost every $\omega$ it holds 
 \begin{equation}\label{stronsync1}
\lim_{n\to \infty}\frac{1}{n}\sum_{i=0}^{n-1}d(f_{\omega}^{i}(x),f_{\omega}^{i}(y))=0.
\end{equation}
 We introduce the following slightly stronger concept, the difference being a ``change of quantifiers''.
 We say that $\varphi$
 is \emph{strongly synchronizing on average} if 
for $\mathbb{P}$-almost every $\omega$ it holds \eqref{stronsync1}
for every $x, y\in X$. 

 The literature suggests that synchronization is often related with the existence of an
invariant graph. An \emph{invariant map} is a measurable map  $\Phi\colon \Omega\to X$ such that 
 \begin{equation}\label{f1omega}
f_{\omega}(\Phi(\omega))=\Phi(\theta(\omega))
\end{equation}
 for $\mathbb{P}$-almost every $\omega$. The graph of an invariant map is called an \emph{invariant graph}.
We observe that invariant maps sometimes provide relevant information about the random product. 
In some cases an invariant map is given explicitly by 
\begin{equation}\label{inversereversed}
\Phi(\omega)\eqdef\lim_{n\to \infty}f_{\theta^{-1}(\omega)}\circ \dots \circ f_{\theta^{-n}(\omega)}(p)
\end{equation}
where the limit exists for $\mathbb{P}$-almost every $\omega$ in $\Omega$ and it is independent of the point $p\in X$.
For instance, in the theory of contracting iterated functions system
the map $\Phi$ is well-defined for every $\omega\in \Omega$ and (see \cite[Section $3.1$]{Hu}) it is the unique invariant map.
Note that in this case the image of $\Phi$ is the
Hutchinson attractor and its distribution is the unique stationary
measure (supported on this attractor)\footnote{Usually this fact is stated using the coordinate map $\Phi^{+}$ as defined in \eqref{reverseorder}. Note that $\Phi^{+}$ is not an invariant map, but it satisfies an invariance equation, that is, $f_{\omega}(\Phi^{+}(\theta(\omega)))=\Phi^{+}(\omega)$ which is the analogous of \eqref{f1omega} for random products over non-invertible transformations.}.
This indicates the importance of
invariant maps for the description of 
dynamical properties of the system.
 To mention some more general results in the i.i.d. case for systems which are non-contracting,
perhaps one of the first in this spirit is due to Letac \cite{Letac}.  The key hypothesis in
\cite{Letac} is that  the map in \eqref{inversereversed} is well-defined  for  
$\mathbb{P}$-almost every $\omega$ in $\Omega$ and it is proved that the distribution of $\Phi$ 
is the unique stationary measure. 
We point out that the property $\Phi$ is well-defined (almost everywhere) is, in
fact, a \emph{consequence} of much weaker topological conditions  such
as the so-called \emph{splitting property} in \cite{DiazMatias} or the
very classical condition of being \emph{contracting on average} in
\cite{Recurrent}.
For further results on the existence of  invariant maps see the comments which follow
Theorem \ref{Dirac}.

In this paper we show that the strong synchronization on average is equivalent to the existence 
of an invariant graph and a unique $\varphi$-invariant measure (see the definition below), see Theorem \ref{Dirac} and Theorem \ref{contractible}.
In general, the uniqueness of a $\varphi$-invariant measure alone does not imply 
any type of synchronization of orbits (see Example \ref{exemple} in Section \ref{sectionsync} ). Thus, it is quite surprising that the existence 
of an invariant map ``forces'' the synchronization. 

The strong synchronization (on average) and the existence of a measurable invariant map
give rise to a globally attracting graph (on average). We show that this attractor is  measurable with respect to the ``past'' of the random product.
As a consequence we can produce a large class of random products satisfying a phenomenon 
called the \emph{vanishing attractor}, see \cite{Misiu2, Misiu} and Corollary \ref{vanishing}. 
 We also use Theorem \ref{Dirac} to state a 
version of Letac's principle \cite{Letac} for non i.i.d. random products on compact metric spaces, see Theorem \ref{Strong}.

  This paper is organized as follows. In Section \ref{s.main} we state the main definitions and the precise 
  statements of our results. In Section \ref{s.existence} we prove Theorem \ref{Dirac} and
   Corollary \ref{vanishing}.  In Section \ref{s.sufficient} we prove Theorems \ref{Breimanmorestrong} and
   \ref{contractible}. Theorem \ref{Breimanmorestrong} is used in the proof of Theorem \ref{contractible} and it can 
   be seen as a weak version of Breiman's ergodic theorem for random products. Finally, Section \ref{s.strong} is devoted to the proof of Theorem \ref{Strong}.

\section{Main results}\label{s.main}

 \subsection{Existence of invariant graphs in the invertible case} \label{existenceinvariantgraphs}
 In the following we take an alternative point of view of random products.
A random product $\varphi$ of maps $f_{\theta^{i}(\omega)}$ on a compact metric space $X$ over a measure-preserving dynamical system
$(\Omega,\mathscr{F},\mathbb{P},\theta)$ induces a \emph{skew product} $F_{\varphi}\colon \Omega\times X\to \Omega\times X$  
defined by
\begin{equation}\label{defiskewproduct}
F_{\varphi}(\omega,x)\eqdef(\theta(\omega),f_{\omega}(x)).
\end{equation}
The map $\theta$ is called the \emph{base map} and the maps $f_{\omega}$ are called the \emph{fiber maps} of the 
skew product $F_{\varphi}$. Note that the second coordinate of $F_{\varphi}^{n}(\omega,x)$ is given by $\varphi(n,\omega,x)$.

  We show that the property of strong synchronization on average implies 
 the existence of a unique (a.e.) invariant map.
 Our approach is to study the space of $\varphi$-\emph{invariant measures}. 
 A probability measure $\mu$ on $\Omega\times X$ 
is called \emph{$\varphi$-invariant} if it is $F_{\varphi}$-invariant and satisfies $(\Pi_{1})_{*}\mu=\mathbb{P}$,
where $\Pi_{1}$ is the projection on the first factor, that is, $\Pi_{1}(\omega,x)=\omega$.
 The space of $\varphi$-invariant measures is denoted by $I_{\mathbb{P}}(\varphi)$. We denote by 
 $\mathcal{M}_{\mathbb{P}}(\Omega\times X)$ the space of probability measures of  $\Omega\times X$ such that 
 $(\Pi_{1})_{*}\mu=\mathbb{P}$. We observe that $I_{\mathbb{P}}(\varphi)\neq \emptyset$ when
 $\Omega$ is a complete metric space, $X$ is a compact 
 metric space and  $\varphi$ is a \emph{continuous random product}, that is, each map $f_{\omega}$ is continuous, see \cite[Theorem 1.5.10]{Arnold}.

In general, for random  products that are synchronizing on average we cannot perform an explicit construction of the invariant map. However, we can obtain additional information of such maps.
 In order to precisely state what we mean by that, we need a few more definitions. 
  Given a random variable $X$ denote by $\sigma(X)$ the $\sigma$-algebra generated by $X$. Given a random product $\varphi$ over an invertible measure-preserving dynamical system 
$(\Omega,\mathscr{F},\mathbb{P},\theta)$ 
  the \emph{past} of $\varphi$ is defined as the sub $\sigma$-algebra
$$
\mathscr{F}^{-}\eqdef\sigma(\omega\mapsto \varphi(n,\theta^{-n}(\omega),x)\colon n\geq 1 , x\in X), 
$$ 
see \cite{Crauel} for details. In other words, $\mathscr{F}^{-}$ is the smallest $\sigma$-algebra that makes
all maps in the above family measurable.

 \begin{mtheorem}\label{Dirac}
 Let $\varphi$ be a random product of maps $f_{\theta^{i}(\omega)}$ on a compact metric space
 $X$ over an invertible measure-preserving dynamical system $(\Omega,\mathscr{F},\mathbb{P},\theta)$ and suppose that 
 $I_{\mathbb{P}}(\varphi)\neq \emptyset$. If $\varphi$ is strongly synchronizing on average, then $I_{\mathbb{P}}(\varphi)$ is a singleton and there exist 
a unique (a.e.) invariant map $\Phi\colon \Omega\to X$. Moreover, the map $\Phi$ is $\mathscr{F}^{-}$-measurable and
 its graph is an attractor for $F_{\varphi}$ in the sense that for $\mathbb{P}$-almost every $\omega$
  $$
 \lim_{n\to \infty} \frac{1}{n}\sum_{i=0}^{n-1}
d(f_{\omega}^{i}(x),\Phi(\theta^{i}(\omega))) = 0,
 $$
  for every $x\in X$.
 
 \end{mtheorem}

Among the few general results about existence of invariant maps is one in 
\cite{Stark}, which we briefly describe now. Assuming that $\theta$ is a homeomorphism 
defined on a compact metric space $\Omega$, $X$ is a complete metric space and considering a random product of 
 Lipschitz maps on $X$ over $\theta$ with \emph{maximal Lyapunov
exponent} (the exponential growth rate of the Lipschitz constant) bounded by a negative constant, it is proved in \cite{Stark} the existence of a unique (a.e.) invariant graph. Note that
the assumption on the Lyapunov exponent immediately implies \emph{exponential synchronization  of orbits}, that is, there is a measurable 
map $c\colon \Omega\to \mathbb{R}^{+}$ and a constant $\rho<1$ such that for $\mathbb{P}$-almost every $\omega$ it holds
$$
d(f_{\omega}^{n}(x), f_{\omega}^{n}(y))\leq c(\omega)\rho^{n}d(x,y)
$$
for every $x,y\in X$ and every $n\geq 0 $.

In Theorem \ref{Dirac} we only assume the existence of a $\varphi$-invariant measure and 
convergence of orbits on average (we do not require exponential convergence). 
Moreover, the maps $f_{\omega}$ are only required to be measurable and the space $\Omega$ is 
any measurable space. 

To finish this discussion, note that \cite{JagerAnagno} considered the case
of a skew product with a minimal homeomorphism on a compact metric
space $\Omega$ in the base and with differentiable fiber maps in
$\mathbb R^n$. 
Considering a random compact invariant set $K\subset\Omega\times\mathbb R^n$
and assuming only that any invariant measure supported on $K$ has a
negative top Lyapunov exponent it is proved that the set $K$ is the finite union of
(possibly multi-valued) continuous invariant graphs. A somewhat more
general setting of random minimal base maps is studied in
\cite{Jager}.

\subsubsection{The Vanishing attractor}\label{ss.vanishing}

In this section we consider  i.i.d. random
  products of maps.
   Namely, let  $(\Sigma_{0},\mathscr{E},\nu)$ be 
a probability space and let $\mathbb{T}$ stand for either $\mathbb{Z}=\{0,\pm{1},\pm{2} ,\dots\}$ or $\mathbb{N}=\{0,1,2,\dots\}$. Consider the space of sequences 
$\Sigma_{0}^{\mathbb{T}}$ endowed 
with the product $\sigma$-algebra $\mathscr{E}^{\mathbb{T}}$ 
and the product measure $\nu^{\mathbb{T}}$.
Denote by $\sigma$ the shift map on the space $\Sigma_{0}^{\mathbb{T}}$, that is, $\sigma$ is defined
 by $\sigma((\omega_{i})_{i\in \mathbb{T}})=(\omega_{i+1})_{i\in \mathbb{T}}$. 
 The shift map is a measure-preserving transformation 
of the probability space $(\Sigma_{0}^{\mathbb{T}},\mathscr{E}^{\mathbb{T}},\nu^{\mathbb{T}})$.
The measure-preserving dynamical systems $(\Sigma,\mathscr{F},\mathbb{P},\sigma)\eqdef(\Sigma_{0}^{\mathbb{Z}},\mathscr{E}^{\mathbb{Z}},\nu^{\mathbb{Z}},\sigma)$ and $(\Sigma^{+},\mathscr{F}^{+},\mathbb{P}^{+},\sigma)\eqdef(\Sigma_{0}^{\mathbb{N}},\mathscr{E}^{\mathbb{N}},\nu^{\mathbb{N}},\sigma)$ are called \emph{two-sided Bernoulli shift} and \emph{one-sided Bernoulli shift}, respectively.

Let $X$ be a compact metric space and consider a measurable map $f\colon \Sigma_{0}^{\mathbb{T}}\times X\to X$ such that 
 the map $f_{\omega}$ (recall that $f_{\omega}(x)\eqdef f(\omega,x)$) depends only on the zeroth coordinate of $\omega\in \Sigma_{0}^{\mathbb{T}}$.
Then the map $\varphi\colon \mathbb{N}\times \Sigma_{0}^{\mathbb{T}}\times X\to X$ defined by
\begin{equation}\label{zzz}
\varphi(n,\omega,x)\eqdef f_{\sigma^{n-1}(\omega)}\circ\cdots\circ f_{\omega}(x)
\end{equation}
 is called an \emph{independent and 
identically distributed (i.i.d.) random product of maps} (over the one-sided or the two-sided Bernoulli shift).

In \cite{Misiu} the authors observe an interesting phenomenon that they call the \emph{mystery 
of the vanishing attractor}. They present an example of an i.i.d. random product 
having a measurable attracting graph which vanishes when the past is forgotten, in the following sense: 
let $f\colon \Sigma\times X\to X$ be a measurable map and 
 consider the i.i.d. random product of maps $\varphi$ over the two-sided Bernoulli shift as 
defined in \eqref{zzz}. 
 Let $\Pi^{+}\colon \Sigma\to \Sigma^{+}$ be the natural projection defined by 
\begin{equation}\label{naturalproj}
\Pi^{+}((\omega_{i})_{i\in\mathbb{Z}})\eqdef(\omega_{i})_{i\in\mathbb{N}}.
\end{equation}
 Associated with $f$ there is a unique map $f^{+}\colon
\Sigma^{+}\times X\to X$ such that $f(\omega,x)=f^{+}(\Pi^{+}(\omega),x)$ for every $(\omega,x)\in \Sigma\times X$.
Hence, we can define a new i.i.d. random product of maps $\varphi^{+}$ over the one-sided Bernoulli shift. 

The passage from $\varphi$ to $\varphi^{+}$ is understood as \emph{forgetting the past} of $\varphi$, 
and the phenomenon of the vanishing attractor described in \cite{Misiu2, Misiu} occurs when $F_{\varphi}$ has an 
attracting graph and $F_{\varphi^{+}}$ does not. This fact may, at first, seem to be counterintuitive since 
attractors are limit sets of forward iterations and any orbit of a point depends only on its
 initial point.

However, below we present a large class of random products having an attractor 
 that vanishes when its past is forgotten. We show that the synchronization of orbits
is a mechanism 
  to produce such a
 phenomenon, because it implies that the invariant graph is 
 measurable with respect to the past of the random product. In this way the vanishing of the attractor will be a natural consequence of the fact that in the i.i.d. case, past and future are independent.

Following \cite{Misiu} we consider a notion of attractor that a priori does not require invariance.
Given an i.i.d. random product $\varphi$ of maps $f_{\sigma^{i}(\omega)}$ over a Bernoulli shift
(two-sided or one-sided) and a measurable map $\Psi\colon\Sigma_{0}^{\mathbb{T}}\to X$ (not necessarily an invariant map)
the \emph{basin of attraction on average} of the set  $\mathrm{graph}\, \Psi$
is defined by
$$
\mathsf{B}_{\Psi}\eqdef\{(\omega,x)\in \Omega\times X\colon \lim_{n\to \infty} \frac{1}{n}\sum_{i=0}^{n-1}
d(f_{\omega}^{i}(x),\Psi(\sigma^{i}(\omega))) = 0 \}.
$$
For every $\omega$, the $\omega$-\emph{section} of $\mathsf{B}_{\Psi}$ is defined by 
$(\mathsf{B}_{\Psi})_{\omega}\eqdef\{x\in X\colon (\omega,x)\in \mathsf{B}_{\Psi}\}$.
The set $\mathrm{graph}\, \Psi$ is called an \emph{attractor on average} for $F_{\varphi}$ if  
 $(\mathsf{B}_{\Psi})_{\omega}$ is a neighbourhood of $\Psi(\omega)$ for $\mathbb{P}$-almost every $\omega$.  If  $(\mathsf{B}_{\Psi})_{\omega}=X$ for $\mathbb{P}$-almost every $\omega$
  then we say that $\mathrm{graph}\, \Psi$ is a \emph{global attractor on average}.

\begin{corollary}[The vanishing attractor]\label{vanishing}
Consider an i.i.d. random product $\varphi$ of maps $f_{\sigma^{i}(\omega)}$ on a compact metric space $X$ over the two-sided 
shift $(\Sigma,\mathscr{F},\mathbb{P},\sigma)$ and suppose that $\varphi$ is strongly synchronizing on average and 
 the maps $f_{\omega}$ have no common fixed point ($\mathbb{P}$-a.e.).
Let $\varphi^{+}$ be the associated i.i.d. random product over the one-sided Bernoulli shift 
$(\Sigma^{+},\mathscr{F}^{+},\mathbb{P}^{+},\sigma)$.
  Then there is no measurable map $\Phi^{+}\colon \Sigma^{+}\to X$ whose graph is an attractor on average for $F_{\varphi^{+}}$.
\end{corollary}

\begin{remark}\emph{
In particular, if $\varphi$ is \emph{strongly synchronizing}, that is, if  for $\mathbb{P}$-almost every $\omega$ it holds 
$$
\lim_{n\to \infty} d(f_{\omega}^{i}(x),f_{\omega}^{i}(y))=0
$$
for every $x, y\in X$, 
then the conclusion of the corollary holds. The class of i.i.d. random products which are strongly synchronizing certainly
includes i.i.d. random products of uniform contracting maps. However, there is a huge class of 
random product of maps without obvious contraction-like properties which are also strongly synchronizing, see \cite[Corollary $2.11$]{Malicet} and \cite[Theorem $3$]{DiazMatias} }.
\end{remark}
  
\subsection{Sufficient conditions for strong synchronization}\label{sectionsync}
In this section we state that random products $\varphi$ on compact metric spaces having 
 an invariant map and a unique $\varphi$-invariant measure is strongly synchronizing on average. To do so, we shall assume that the induced skew product is continuous.

 We start by stablishing a version of the Breiman ergodic theorem \cite{Breiman} for random products having a unique $\varphi$-invariant measure. Throughout, if $\Omega$ is a compact metric space then the $\sigma$-algebra 
 considered in any measure-preserving dynamical system $(\Omega,\mathscr{F},\mathbb{P},\theta)$ is the Borel $\sigma$-algebra of $\Omega$. The convergence of a sequence of Borel probability measures is always considered in the weak$*$-topology.

\begin{mtheorem}[Strong Ergodic theorem]\label{Breimanmorestrong}
Let $\Omega$ be a compact metric space and consider a random product of maps $\varphi$ on a compact metric space
 $X$ over an invertible and ergodic measure-preserving dynamical system $(\Omega,\mathscr{F},\mathbb{P},\theta)$. Suppose that the skew product $F_{\varphi}$ induced by $\varphi$ is continuous and there is a unique $\varphi$-invariant measure $\mu$. Then 
for $\mathbb{P}$-almost every $\omega$ and every $x\in X$ we have 
$$
\lim_{n\to \infty}\frac{1}{n}\sum_{i=0}^{n-1}\delta_{F_{\varphi}^{i}(\omega,x)}=\mu.
$$ 

\end{mtheorem}

\begin{remark}\emph{Let $\varphi$ be an i.i.d. random product of maps $f_{\sigma^{i}(\omega)}$ on a compact metric space $X$
over the two-sided shift $(\Sigma,\mathscr{F},\mathbb{P},\sigma)$, as defined in Section \ref{ss.vanishing}. The family of sequences of random variables $(X^{x}_{n})_{x\in X}$ defined by $X_{n}^{x}(\omega)=f_{\omega}^{n}(x)$ is a family of Markov chains with common transition 
probability, for more details see \cite[Theorem 2.1.4]{Arnold}. 
If there is a unique $\varphi$-invariant measure $\mu$ then this family of Markov chains has a unique stationary measure 
given by $m=\Pi_{2*}\mu$, where $\Pi_{2}$ is the natural projection on the second factor, that is, $\Pi_{2}(\omega,x)=x$ (for details see \cite[Theorem 2.1.8]{Arnold}). Therefore, for every $x$ we can apply the
 Breiman ergodic theorem \cite{Breiman} to obtain  a set $\Sigma^{x}$ with $\mathbb{P}$-full measure such that 
$$
\lim_{n\to \infty}\sum_{i=0}^{n-1}\phi(f_{\omega}^{i}(x))= \int \phi\, dm
$$
for every $\omega\in \Sigma^{x}$ and every continuous function $\phi\colon \Sigma\to \mathbb{R}$.}

\emph{
Let us observe that if we apply Theorem \ref{Breimanmorestrong} to the setting above then we get a slightly stronger statement 
than the one obtained applying the Breiman ergodic theorem. Indeed,
 by Theorem \ref{Breimanmorestrong} there is a set $\Sigma'$
of $\mathbb{P}$-full measure such that for every $\omega\in \Sigma'$  it holds
$$
\lim_{n\to \infty}\frac{1}{n}\sum_{i=0}^{n-1}\phi(f_{\omega}^{n}(x))=\lim_{n\to \infty}\int \phi\circ \Pi_{2}\,\, d\left(\frac{1}{n}\sum_{i=0}^{n-1}\delta_{F^{i}(\omega,x)}\right)
=\int \phi\circ \Pi_{2} \, d\mu=\int \phi \, dm
$$
for every $x\in X$ and every continuous function $\phi\colon X\to \mathbb{R}$.}

\end{remark}

We now stablish the main result of this section

\begin{mtheorem}\label{contractible}
Let $\Omega$ be a compact metric space and let $\varphi$ be a random product of maps on a compact metric space
 $X$ over an invertible and ergodic measure-preserving dynamical system $(\Omega,\mathscr{F},\mathbb{P},\theta)$. Suppose that the skew product $F_{\varphi}$ induced by $\varphi$ is continuous and there are an invariant graph
 and a unique $\varphi$-invariant measure. Then $\varphi$ is strongly synchronizing on average.
\end{mtheorem}

The following example shows that the uniqueness of the $\varphi$-invariant measure alone does not 
imply synchronization on average.
\begin{example}\label{exemple}\emph{
Let $\Omega=\{1,2\}$ endowed with the discrete $\sigma$-algebra $\mathscr{D}$ and the probability measure 
$\mathbb{P}=\frac{1}{2}\delta_{1}+\frac{1}{2}\delta_{2}$.  Let $\theta\colon \Omega\to \Omega$ be the 
measure-preserving transformation of $(\Omega,\mathscr{D},\mathbb{P})$ defined by 
$
\theta(1)=2 \quad \mbox{and}\quad \theta(2)=1.
$}

\emph{
Let $X=\{1,2\}$ endowed with the discrete topology and let $f\colon \Omega \times X\to X$ be defined by 
$f(1,1)=2$, $f(2,2)=2$, $f(1,2)=1$ and $f(2,1)=1$. Then, the random product $\varphi$ on $X$ over
$(\Omega,\mathscr{D},\mathbb{P}, \theta)$, whose generator is $f$, has a unique $\varphi$-invariant measure 
and it is not synchronizing on average.  }
\end{example}

\subsection{Letac's principle for a certain class of random products}
In this section we stablish a version of Letac's principle for non i.i.d. random products
on a compact metric space. To recall the Letac's principle \cite{Letac},
 let $\varphi$ be an i.i.d. continuous random product of maps $f_{\sigma^{i}(\omega)}$ on a metric space $X$
 over the one-sided Bernoulli shift $(\Sigma^{+},\mathscr{F}^{+},\mathbb{P}^{+},\sigma)$ (recall definition 
 in Section \ref{ss.vanishing}). Suppose that 
 for $\mathbb{P}^{+}$-almost every $\omega$ the limit of reversed order iterates
\begin{equation}\label{reverseorder}
\Phi^{+}(\omega)\eqdef\lim_{n\to \infty}f_{\omega}\circ \dots \circ f_{\sigma^{n}(\omega)}(p)
\end{equation}
exists and does not depend on $p\in X$. With this assumption it is proved in \cite{Letac} that
the family of Markov chains $X_{n}^{x}(\omega)=f_{\omega}^{n}(\omega)$, $x\in X$ (with common probability transition) has a unique stationary measure.

The next result can be seen as a version of Letac's principle for non i.i.d. random products. In fact, we assume a similar convergence of the
reversed order iterates as in \eqref{reverseorder}, but in a uniform way, and we obtain synchronization 
on average in a uniform way, too. In particular,  Theorem \ref{Dirac} gives us the uniqueness of the invariant measure.

\begin{mtheorem}\label{Strong} Let $\varphi$ be a random product of maps $f_{\theta^{i}(\omega)}$ on a compact metric space
 $X$ over an invertible measure-preserving dynamical system $(\Omega,\mathscr{F},\mathbb{P},\theta)$. 
  Then the following facts are equivalent
 \begin{itemize}
 \item[(1)]
  $
\displaystyle \lim_{n \rightarrow \infty} \mathrm{diam}\,(f_{\theta^{-1}(\omega)}\circ \cdots\circ f_{\theta^{-n}(\omega)}(X))=0
\quad \mbox{for}\,\, \mathbb{P}\mbox{-almost every}\,\, \omega.
 $
\item[(2)]
 $
\displaystyle \lim_{n\to\infty}\frac{1}{n}\sum_{i=1}^{n} \mathrm{diam}\,(f_{\omega}^{i}(X))= 0\quad \mbox{for}\,\, \mathbb{P}\mbox{-almost every}\,\, \omega.
 $

 \end{itemize}

\end{mtheorem}

Theorem \ref{Strong} can be applied, for instance, to obtain 
the synchronization on average of a certain class of Markovian random products. In \cite[Theorem $4.1$]{DiazMatias} is presented 
a large a class of Markovian random products for which item (1) of the Theorem \ref{Strong} is satisfied.

\begin{remark}\emph{
We observe that, in general, it is not possible to prove that $(1)$ implies $\mathrm{diam}\,(f_{\omega}^{n}(X))\to 0$, 
see \cite[Section $6$]{Italo}.} 

\end{remark}

\section{Existence of measurable invariant graphs}\label{s.existence}
In this section we prove Theorem \ref{Dirac} and Corollary \ref{vanishing}.
Let $(X,d)$ be a compact metric space. We first introduce a well-known tool 
to detect Dirac measures on $X$.

\begin{definition}
For any Borel probability measure $m$ on $X$ let
$$
D(m)\eqdef \iint d(x,y)\,dm(x)\,dm(y).
$$
\end{definition}

\begin{lemma}\label{DDirac}
If $D(m)=0$ then there is $p\in X$ such that $m=\delta_{p}$. 
\end{lemma}
\begin{proof}
If $D(m)=0$, then by definition 
we have 
$$
d(x,y)=0\quad \mbox{for}\quad m\times m \mbox{\,-almost every} \quad (x,y)\in X\times X,
$$
where $m\times m$ is the product measure on $X \times X$. This implies that 
the support of the measure $m\times m$ is a subset of the set $\mathrm{Diag}(X)=\{(x,x)\colon x\in X\}$.
Suppose that $\mathrm{supp}\, m$ is not a singleton. Then there are $a$ and $b$ in $\mathrm{supp}\, m$
with $a\neq b$. Let $V_{1}$ and $V_{2}$ be open disjoints neighbourhoods of $a$ and $b$, respectively. Since $a$ and $b$ belong to the support of $m$, we get $m\times m (V_{1}\times V_{2})>0$.
Note also that, by the disjointness, we have $(V_{1}\times V_{2})\cap \mathrm{Diag}(X)=\emptyset$, which 
contradicts the fact that the support of $m\times m$ is a subset of $\mathrm{Diag}(X)$. The proof 
of the lemma is now complete.
\end{proof}

\subsection{Proof of Theorem \ref{Dirac}}
 
By hypothesis $I_{\mathbb{P}}(\varphi)\neq \emptyset$.  We start by proving that $I_{\mathbb{P}}(\varphi)$ is a singleton.
Let $\mu$ be any probability measure in $I_{\mathbb{P}}(\varphi)$ and consider its disintegration $\omega\mapsto\mu_{\omega}$
 with respect to $\mathbb{P}$.
The definition of $D$ and the synchronization property imply that for $\mathbb{P}$-almost every $\omega$ in $\Omega$, 
$$
\lim_{n\to \infty}\frac{1}{n}\sum_{j=0}^{n-1}D(f_{\omega*}^{j}\mu_{\omega})=\lim_{n\to \infty}
\iint \frac{1}{n}\sum_{j=0}^{n-1}
d(f_{\omega}^{n}(x),f_{\omega}^{n}(y))\,d\mu_{\omega}(x)\,d\mu_{\omega}(y)=0.
$$
Hence, using the dominated convergence theorem, we have 
\begin{equation}\label{synconaverage}
\lim_{n\to \infty}\int \frac{1}{n}\sum_{j=0}^{n-1}D(f_{\omega*}^{j}\mu_{\omega})\,d\mathbb{P}=0.
\end{equation}
Recall that a measure $\mu$ on $\mathcal{M}_{\mathbb{P}}(\Omega\times X)$ is $\varphi$-invariant if and only if $\mu_{\theta(\omega)}=f_{\omega*}\mu_{\omega}$  for $\mathbb{P}$-almost every $\omega$ in $\Omega$, see \cite[Theorem 1.4.5]{Arnold}. In particular, for every 
$n\geq 1$ the equality
$
\mu_{\theta^{n}(\omega)}=f_{\omega*}^{n}\mu_{\omega}
$ 
holds
for $\mathbb{P}$-almost every $\omega$, and as a consequence,
\begin{equation}\label{17072017}
\begin{split}
\int \frac{1}{n}\sum_{j=0}^{n-1}D(f_{\omega*}^{j}\mu_{\omega})\,d\mathbb{P}
=\int \frac{1}{n}\sum_{j=0}^{n-1}D(\mu_{\theta^{j}(\omega)})\,d\mathbb{P}
=&\frac{1}{n}\sum_{j=0}^{n-1}\int D(\mu_{\theta^{j}(\omega)})\,d\mathbb{P}.
\end{split}
\end{equation}
Note that, by the invariance of $\theta$, we have   
$$
\int D(\mu_{\theta^{j}(\omega)})\,d\mathbb{P}=\int D(\mu_{\omega})\,d\mathbb{P}
$$
for every $j\geq 1$. Thus it follows from \eqref{synconaverage} and \eqref{17072017} that $D(\mu_{\omega})=0$ for $\mathbb{P}$-almost every $\omega$ in $\Omega$.
 Applying Lemma \ref{DDirac} we get that 
$\mu_{\omega}$ is a delta Dirac measure for $\mathbb{P}$-almost every $\omega$, that is, there is a set $\Lambda\subset\Omega$ with $\mathbb{P}$-full measure and a map $\Phi\colon \Lambda\to X$ such that 
$$
\mu_{\omega}=\delta_{\Phi(\omega)}
$$
for every $\omega\in \Lambda$.

This shows that the disintegration of any $\varphi$-invariant measure is atomic. This fact implies 
that $I_{\mathbb{P}}(\varphi)$ is a singleton. Indeed, let $\mu$ and $\nu$ be two $\varphi$-invariant measures.
Then there are maps $\Phi, \Psi\colon \Omega\to X$  such that $\omega\mapsto \delta_{\Phi(\omega)}$ and 
$\omega\mapsto \delta_{\Psi(\omega)}$ are the disintegration of $\mu$ and $\nu$, respectively.
The probability measure $\frac{\mu+\nu}{2}$ is also a $\varphi$-invariant measure and then there exists a map 
$\Gamma\colon \Omega\to X$ such that the disintegration of $\frac{\mu+\nu}{2}$ is given by $\omega\mapsto \delta_{\Gamma(\omega)}$.
On the other hand,  $\omega\mapsto \frac{\delta_{\Phi(\omega)}+\delta_{\Psi(\omega)}}{2}$ is also a 
disintegration of $\frac{\mu+\nu}{2}$. By the uniqueness (a.e.) of the disintegration of a probability measure we have 
$$
 \frac{\delta_{\Phi(\omega)}+\delta_{\Psi(\omega)}}{2}=\delta_{\Gamma(\omega)}
$$
for $\mathbb{P}$-almost every $\omega$ in $\Omega$, which implies that $\Gamma(\omega)=\Phi(\omega)=\Psi(\omega)$
for $\mathbb{P}$-almost every $\omega$ in $\Omega$. Therefore, $I_{\mathbb{P}}(\varphi)$ is a singleton.

Let $\mu$ be the unique $\varphi$-invariant measure and let $\Phi\colon \Omega\to X$ be a map such that 
$\omega\mapsto \delta_{\Phi(\omega)}$ is the disintegration of $\mu$.
We claim that $\Phi$ is an invariant map. We start by proving that $\Phi$ is measurable: for every Borel 
measurable set $R\subset X$ we have $\gamma_{R}^{-1}(\{1\})=\Phi^{-1}(R)$, where $\gamma_{R}$ is the map
 $\omega\mapsto \mu_{\omega}(R)$. By the definition of disintegrations 
 we have that $\gamma_{R}$ is measurable, which implies that $\Phi$ is also measurable.
 
 To prove  that $\Phi$ is an invariant map recall that the $\varphi$-invariance of $\mu$ implies $\mu_{\theta(\omega)}=f_{\omega*}\mu_{\omega}$ for $\mathbb{P}$-almost every $\omega$ in $\Omega$. Hence
 we have 
\begin{equation}\label{a.t}
\delta_{\Phi(\theta(\omega))}=\mu_{\theta(\omega)}=
f_{\omega*}\mu_{\omega}=f_{\omega*}\delta_{\Phi(\omega)}=\delta_{f_{\omega}(\Phi(\omega))}
\end{equation}
for $\mathbb{P}$-almost every $\omega$ in $\Omega$, which implies that $f_{\omega}(\Phi(\omega))=\Phi(\theta(\omega))$  for $\mathbb{P}$-almost every $\omega$ in $\Omega$. The uniqueness (a.e.) of the invariant map is proved observing that 
every invariant map determines a $\varphi$-invariant measure, as in \eqref{a.t}.

We now prove that the graph of $\Phi$ is a global attractor on average. First note that
 the invariance of $\Phi$ implies that for every $n\geq 0$
$$
\Phi(\theta^{n}(\omega))=f_{\omega}^{n}(\Phi(\omega)),
$$
for $\mathbb{P}$-almost every $\omega$ in $\Omega$. Thus, it follows from 
the synchronization property that for $\mathbb{P}$-almost every $\omega$ in $\Omega$ we have
\begin{equation}\label{attractoraverag2}
	\lim_{n\to \infty}\frac{1}{n}\sum_{i=0}^{n-1}d(f_{\omega}^{i}(x),\Phi(\theta^{i}(\omega)))=0
\end{equation} 
for every $x\in X$.

Finally, we prove that $\Phi$ is $\mathscr{F}^{-}$-measurable. From \eqref{attractoraverag2} and the
dominated convergence theorem we have 
\begin{equation}\label{intatractor}
	\lim_{n\to \infty}\int\frac{1}{n}\sum_{i=0}^{n-1} d(f_{\omega}^{i}(x),\Phi(\theta^{i}(\omega))\,d\mathbb{P}=0,
\end{equation}
for every $x\in X$. Take any $x\in X$, by the invariance of $\theta$ it follows from \eqref{intatractor} that
\begin{equation}\label{codingmap}
	\lim_{n\to \infty}\frac{1}{n}\sum_{i=0}^{n-1} \int  d( f^{i}_{\theta^{-i}(\omega)}(x),\Phi(\omega))\,d\mathbb{P}=0.
\end{equation}
Since $d( f^{i}_{\theta^{-i}(\omega)}(x),\Phi(\omega))\geq 0$ for every $i \geq 0$, we obtain from \eqref{codingmap}
that there exists a subsequence $(n_{k})_{k}$ such that 
$$
\lim_{k\to \infty}\int  d( f^{n_{k}}_{\theta^{-n_{k}}(\omega)}(x),\Phi(\omega))\,d\mathbb{P}= 0.
$$
Thus, passing to another subsequence if necessary we have

$$
\Phi(\omega)=\lim_{k\to \infty}f^{n_{k}}_{\theta^{-n_{k}}(\omega)}(x),
$$
for $\mathbb{P}$-almost every $\omega$ in $\Omega$,
which implies that $\Phi$ is $\mathscr{F}^{-}$-measurable.
The proof of the theorem is now complete.
\hfill \qed

\begin{remark}\emph{
Note that in the above proof we needed to prove first that $\Phi$ is measurable because to prove that it is 
 $\mathscr{F}^{-}$-measurable we take an integral involving this map, see \eqref{intatractor}.}

\end{remark}

\subsection{Proof of Corollary \ref{vanishing}}

By Theorem \ref{Dirac} there exists a $\mathscr{F}^{-}$-measurable map $\Phi\colon\Sigma\to X$ whose graph is a global
attractor on average.

Suppose that there is a measurable map $\Phi^{+}\colon \Sigma^{+}\to X$ whose graph 
is an attractor on average. Let $\Pi^{+}\colon \Sigma\to \Sigma^{+}$ be the natural projection defined 
in \eqref{naturalproj}.

 \begin{claim}
 $\Phi(\omega)= \Phi^{+}\circ\Pi^{+}(\omega)$ for $\mathbb{P}$-almost every $\omega$.
 \end{claim}

\begin{proof} 

By hypothesis, the graph of $\Phi^{+}$ is an attractor on average for $F_{\varphi^{+}}$. 
Noting that $\Pi^{+}_{*}\mathbb{P}=\mathbb{P}^{+}$, we get that the graph of 
$\Phi^{+}\circ\Pi^{+}$ is an attractor on average for $F_{\varphi}$. Observe that for every $x$ and $\omega$ we have
\begin{equation}\label{tarde}
d(\Phi(\theta^{i}(\omega)),\Phi^{+}\circ \Pi^{+}(\theta^{i}(\omega)))\leq d(f^{i}_{\omega}(x),\Phi(\theta^{i}(\omega)))+ d(f^{i}_{\omega}(x),\Phi^{+}\circ \Pi^{+}(\theta^{i}(\omega))),
\end{equation}
for every $i\geq 0$. Since the graph of $\Phi$ is a global attractor on average, for $\mathbb{P}$-almost 
every $\omega$ we can take a point $x(\omega)\in X$ such that 
$$
x(\omega)\in \rho(\Phi)_{\omega}\cap  \rho(\Phi^{+}\circ \Pi^{+})_{\omega}.
$$
Applying \eqref{tarde} to the point $x(\omega)$ we obtain that for $\mathbb{P}$-almost every $\omega$
$$
\lim_{n\to \infty}\frac{1}{n}\sum_{i=0}^{n-1}d(\Phi(\theta^{i}(\omega)),\Phi^{+}\circ \Pi^{+}(\theta^{i}(\omega))=0.
$$
Using the dominated convergence theorem and the invariance of $\theta$ we conclude that
\begin{equation}\label{semergodicidade}
\begin{split}
0&=\lim_{n\to \infty}\int \frac{1}{n}\sum_{i=0}^{n-1}d(\Phi(\theta^{i}(\omega)),\Phi^{+}\circ \Pi^{+}(\theta^{i}(\omega))\,d\mathbb{P}(\omega)\\
&=\lim_{n\to \infty} \frac{1}{n}\sum_{i=0}^{n-1}\int d(\Phi(\omega),\Phi^{+}\circ \Pi^{+}(\omega))\,d\mathbb{P}(\omega)=
\int d(\Phi(\omega),\Phi^{+}\circ \Pi^{+}(\omega))\,d\mathbb{P}(\omega).
\end{split}
\end{equation}
Since $d(\Phi(\omega),\Phi^{+}\circ \Pi^{+}(\omega))\geq 0$ for every $\omega$, we obtain from \eqref{semergodicidade} 
that $\Phi(\omega)=\Phi^{+}\circ \Pi^{+}(\omega)$ for $\mathbb{P}$-almost every $\omega$, ending the proof
 of the claim.

\end{proof}

In order to conclude the proof of Corollary \ref{vanishing}, we consider the \emph{future} of 
the random product $\varphi$ defined by 
$$
\mathscr{F}^{+}\eqdef\sigma(\omega\mapsto \varphi(n,\theta^{n}(\omega),x) \colon n\geq 0, x\in X).
$$
Note that $\Phi^{+}\circ \Pi^{+}$ 
is $\mathscr{F}^{+}$-measurable. On the other hand, $\Phi$ is $\mathscr{F}^{-}$-measurable.
Since $\mathscr{F}^{+}$ and $\mathscr{F}^{-}$ are independent $\sigma$-algebras we obtain that $\Phi$ is 
constant $\mathbb{P}$-almost everywhere. Indeed, given a measurable set $A\in X$,
by independence we have 
\[
\begin{split}
\mathbb{P}(\Phi^{-1}(A))&=\mathbb{P}(\Phi^{-1}(A)\cap(\Phi^{+}\circ \Pi^{+})^{-1}(A))\\
&=
\mathbb{P}(\Phi^{-1}(A))\mathbb{P}((\Phi^{+}\circ \Pi^{+})^{-1}(A))=\mathbb{P}(\Phi^{-1}(A))^{2},
\end{split}
\]
which implies that $\mathbb{P}(\Phi^{-1}(A))=0$ or $\mathbb{P}(\Phi^{-1}(A))=1$. Thus, the $\sigma$-algebra 
generated by $\Phi$ coincides (mod $0$) with the trivial $\sigma$-algebra $\{\emptyset,\Omega\}$,
which implies that $\Phi$ is $\mathbb{P}$-a.e. constant, say, $\Phi(\omega)=p$ for
$\mathbb{P}$-almost every $\omega$ for some $p\in X$. Since $\Phi$ is an invariant map, we obtain $f_{\omega}(p)=p$ for $\mathbb{P}$-almost 
every $\omega$, which contradicts the hypothesis. Therefore, there is no measurable map $\Phi^{+}$ whose 
graph is an attractor on average, ending the proof of the corollary.
 \hfill \qed

\section{Sufficient condition for synchronization on average} \label{s.sufficient}
In this section we prove Theorems \ref{Breimanmorestrong} and \ref{contractible}.

\subsection{Proof of Theorem \ref{Breimanmorestrong}}
Let $\Omega_{\theta}$ be a subset 
 of $\Omega$ such that for every $\omega\in \Omega_{\theta}$ we have 
 \begin{equation}\label{deltadeltadelta}
  \lim_{n\to \infty}\frac{1}{n}\sum_{i=0}^{n-1}\delta_{\theta^{i}(\omega)}= \mathbb{P}.
  \end{equation}
 Since $\theta$ is ergodic we deduce that $\mathbb{P}(\Omega_{\theta})=1$.
Let $F$ be the skew product associated with $\varphi$, recall \eqref{defiskewproduct}.
 We claim 
that for $(\omega,x)\in \Omega_{\theta}\times X$ it holds 
$$
\lim_{n\to \infty}\frac{1}{n}\sum_{i=0}^{n-1}\delta_{F^{i}(\omega,x)} = \mu,
$$
 where $\mu$ is the unique $\varphi$-invariant measure.
To see why this is so, let $\Pi_{1}\colon \Omega\times X\to\Omega $ be the natural projection on the first factor.
  Then for 
 every $(\omega,x)$ it holds
 $$
  \Pi_{1*}\left(
 \frac{1}{n}\sum_{i=0}^{n-1}\delta_{F^{i}(\omega,x)} \right)= \frac{1}{n}\sum_{i=0}^{n-1}\delta_{\theta^{i}(\omega)}.
 $$
 Hence, it follows from \eqref{deltadeltadelta} and the continuity of $\Pi_{1*}$ (in the weak$*$- topology) that 
for every $(\omega,x)\in \Omega_{\theta}\times X$ any accumulation point $\mu'$
 of $\frac{1}{n}\sum_{i=0}^{n-1}\delta_{F^{i}(\omega,x)}$ satisfies $\Pi_{1*}\mu'=\mathbb{P}$.  Since any
 accumulation point of $\frac{1}{n}\sum_{i=0}^{n-1}\delta_{F^{i}(\omega,x)}$
 is a $F$-invariant measure then it follows 
 from the uniqueness of $\varphi$-invariant measures that  
 $$
 \lim_{n\to \infty}\frac{1}{n}\sum_{i=0}^{n-1}\delta_{F^{i}(\omega,x)}= \mu, 
 $$
for every $(\omega,x)\in \Omega_{\theta}\times X$. 

\hfill \qed

 \subsection{Proof of Theorem \ref{contractible}}
Suppose that there exists a unique $\varphi$-invariant measure $\mu$ and assume the existence of a measurable 
invariant map $\Phi\colon \Omega\to X$.  Let $g\colon \Omega\times X\to \mathbb{R}$ be the map defined by
 $$
 g(\omega,x)\eqdef d(x,\Phi(\omega)).
 $$
 Since $\omega\mapsto d(x,\Phi(\omega))$ 
is measurable for every $x$ and $x\mapsto d(x,\Phi(\omega))$ is continuous for 
every $\omega$, Carath\' eodory’s theorem
implies that the function $g$ is $\mathscr{F}\otimes\mathscr{B}$-measurable, 
see \cite{Crauel}. The following preliminary result is the key step to the proof of Theorem \ref{contractible}.

\begin{lemma}\label{principall}
For $\mathbb{P}$-almost every $\omega$ in $\Omega$ we have 
$$
\lim_{n\to \infty} \int g \, d\mu_{n}^{\omega,x}=0
$$
for every $x\in X$,
where $\mu_{n}^{\omega,x}= \frac{1}{n}\sum_{i=0}^{n-1}\delta_{F^{i}(\omega,x)}$.
\end{lemma}
 \begin{proof}

 Since $\Omega$ is a compact metric 
space, given $k>0$ it follows from Lusin's theorem that 
there exists a compact set $F_{k}\subset\Omega$ such that $\Phi_{|F_{k}}$ is a continuous 
map and $\mathbb{P}(F_{k}^{c})<\frac{1}{k}$. Note that 
$g$ is continuous on $F_{k}\times X$. 

\begin{claim}\label{Alexandrov}
$g \mathds{1}_{F_{k}\times X}$ is upper semi-continuous.
\end{claim}
\begin{proof}
Given $\alpha>0$ we need to see that the set $A_{\alpha}=\{(\omega,x)\in \Omega\times X\colon g(\omega,x)<\alpha\}$ is an 
open set. If $A_{\alpha}$ is not empty, given a point $(\xi,y)\in A_{\alpha}$, since $g_{|F_{k}\times X}$ is continuous there is an open neighbourhood $V$ of $(\xi,y)$ such that 
$$
g \mathds{1}_{F_{k}\times X}(\omega,x)=g(\omega,x)<\alpha\quad \mbox{for every}\quad (\omega,x)\in V\cap F_{k}\times X.
$$
On the other hand, if $(\omega,x)\notin F_{k}\times X$ then  $g \mathds{1}_{F_{k}\times X}(\omega,x)=0$, proving 
the claim.
\end{proof}
By Theorem \ref{Breimanmorestrong}, there exists a set $\Lambda\subset \Omega$ with $\mathbb{P}$-full measure 
 such that 
 $
 \lim_{n\to \infty}\mu_{n}^{\omega,x}= \mu
 $
 for every $(\omega,x)\in \Lambda\times X$. Then, it follows from
Alexandrov's theorem and Claim \ref{Alexandrov} that 
\begin{equation}\label{limsup2}
\limsup_{n\to \infty} \int g\mathds{1}_{F_{k}\times X}\, d\mu_{n}^{\omega,x}\leq  \int g\mathds{1}_{F_{k}\times X}\,d\mu\leq 
\int g \, d\mu
\end{equation}
for every $(\omega,x)\in \Lambda\times X$.

Now let us compute $\int g\, d\mu$. The uniqueness of $\mu$ and the invariance of $\Phi$ implies 
that the disintegration of $\mu$ with respect to $\mathbb{P}$ is given by $\omega\mapsto \delta_{\Phi(\omega)}$. 
Hence,
\[
\begin{split}
\int g(\omega,x)\,d\mu(\omega,x)
&=
\int_{\Omega}\left(\int_{X}g(\omega,x)d\delta_{\Phi(\omega)}(x)\right)
d\mathbb{P}(\omega)\\
&=\int_{\Omega}g(\omega,\Phi(\omega))
d\mathbb{P}(\omega)=0.
\end{split}
\]
Thus, by \eqref{limsup2} we deduce that $\limsup_{n\to \infty} \int g\mathds{1}_{F_{k}\times X}\, d\mu_{n}^{\omega,x}=0$ 
for every $(\omega,x)\in \Lambda\times X$.

Since $X$ is  a compact metric space, we have
$L\eqdef\sup_{x,y\in X} d(x,y)< \infty$. Hence for every $(\omega,x)\in \Lambda\times X$ we have 
\begin{equation}\label{limsup}
\begin{split}
\limsup_{n\to \infty} \int g\, d\mu_{n}^{\omega,x}
&= \limsup_{n\to \infty}\left( \int_{F_{k}\times X} g\, d\mu_{n}^{\omega,x}+
\int_{F_{k}^{c}\times X}g\, d\mu_{n}^{\omega,x}\right)\\
&\leq\limsup_{n\to \infty}\int_{F_{k}\times X} g\, d\mu_{n}^{\omega,x}+\limsup_{n\to \infty} \int_{F_{k}^{c}\times X} g\, d\mu_{n}^{\omega,x}\\
&\leq L \limsup_{n\to \infty}\mu_{n}^{\omega,x}(F_{k}^{c}\times X).
\end{split}
\end{equation}

\begin{claim}\label{Fk}
There exists a set $\Lambda'\subset \Omega$ with $\mathbb{P}$-full measure such that 
$$
\lim_{n\to \infty}\mu_{n}^{\omega,x}(F_{k}^{c}\times X)=\mathbb{P}(F_{k}^{c}) 
$$
for every $\omega\in \Lambda'$, every $x\in X$ and every $k\geq 1$.
\end{claim}

\begin{proof}

It follows from the definition of $\mu_{n}^{\omega,x}$ that for every $(\omega,x)\in \Omega\times X$ and every $k\geq 1$ we have
$$
\mu_{n}^{\omega,x}(F_{k}^{c}\times X)=\frac{1}{n}\sum_{i=0}^{n-1}\delta_{\theta^{i}(\omega)}(F_{k}^{c}).
$$
By Birkhoff's ergodic theorem there is a set $\Lambda_{k}$ with 
$\mathbb{P}(\Lambda_{k})=1$ and such that for every $\omega\in \Lambda_{k}$ it holds 
$$
\lim_{n\to \infty}\frac{1}{n}\sum_{i=0}^{n-1}\delta_{\theta^{i}(\omega)}(F_{k}^{c})=\mathbb{P}(F_{k}^{c}).
$$
Thus, defining $\Lambda'\eqdef \cap_{k\geq 1} \Lambda_{k}$ we have 
$$
\lim_{n\to \infty}\mu_{n}^{\omega,x}(F_{k}^{c}\times X)=\mathbb{P}(F_{k}^{c})
$$
for every $(\omega,x)\in \Lambda'\times X$ and every $k\geq 1$, ending the proof of the claim.
\end{proof}

Now, setting $\Omega'\eqdef\Lambda'\cap\Lambda$, it follows from \eqref{limsup} and Claim \ref{Fk} that
$$
\limsup_{n\to \infty} \int g\, d\mu_{n}^{\omega,x}\leq L \,\mathbb{P}(F_{k}^{c})\leq L\frac{1}{k}
$$
for every $(\omega,x)\in \Omega'\cap X$ and
 every $k\geq 1$. Hence, from \eqref{limsup} we have 
$$
\lim_{n\to \infty} \int g\, d\mu_{n}^{\omega,x}=0
$$
for every $(\omega,x)\in \Omega'\cap X$.
The proof of the lemma is now complete.
\end{proof}

We now conclude the proof of the theorem. First, note that for
 every $\omega\in \Omega$ and every pair $x,y\in X$ we have
$$
\frac{1}{n}\sum_{i=0}^{n-1}d(f_{\omega}^{i}(x),f_{\omega}^{i}(y))\leq \frac{1}{n}\sum_{i=0}^{n-1}d(f^{j}_{\omega}(x),\Phi(\theta^{j}(\omega)))+ \frac{1}{n}\sum_{i=0}^{n-1}d(f^{j}_{\omega}(y),\Phi(\theta^{j}(\omega))),
$$
and it follows from the definition of $\mu_{n}^{\omega,z}$ that
$$
\int g \, d\mu_{n}^{\omega,z}=\frac{1}{n}\sum_{i=0}^{n-1}d(f^{j}_{\omega}(z),\Phi(\theta^{j}(\omega))),
$$
for every $\omega$ and every $z\in X$. Then from Lemma \ref{principall} we deduce that
for $\mathbb{P}$-almost every $\omega$ in $\Omega$
$$
\lim_{n\to \infty}\frac{1}{n}\sum_{i=0}^{n-1}d(f_{\omega}^{i}(x),f_{\omega}^{i}(y))=0,
$$
for every $x,y\in X$.
The proof of the theorem is now complete.

\hfill \qed
\section{Proof of Theorem \ref{Strong}}\label{s.strong}
We first prove that (1)$\Rightarrow$(2). For that, consider the sequence of measurable functions $g_n:\Omega \rightarrow [0,+\infty)$ defined by
$
g_n(\omega)\eqdef \mathrm{diam}\,(f_{\omega}^{n}(X))  
$, $n\geq 1$, and observe that,
\[
\begin{split}
g_{n+1}(\omega) &= \mathrm{diam}\,(f_{\theta^{n}(\omega)}\circ \cdots\circ f_{\omega}(X)) \\
&= \mathrm{diam}\, (f_{\theta^{n+1}(\omega)}\circ \cdots\circ f_{\theta(\omega)}(f_{\omega}(X)))\\
&\leq g_n(\theta(\omega)),
\end{split}
\]
for every $\omega\in\Omega$ and every $n\geq1$. Hence, for $m,k\in \mathbb{N}$ we have $g_{m+k}(\omega) \leq g_{m}(\theta^k(\omega))$. Define $u_n(\omega) = \displaystyle\sum_{j=1}^n g_j(\omega)$, and note that
$$
u_{m+n} =\displaystyle\sum_{j=1}^{m+n} g_j
= u_m + \displaystyle\sum_{j=m+1}^{m+n} g_j
\leq u_m + \displaystyle\sum_{i=1}^{n} g_i \circ \theta^m= u_m + u_n \circ \theta^m,
$$
which implies that the sequence $(u_n)_n$ is subadditive. Thus, by Kingman's subadditive ergodic theorem
 the sequence $(u_n/n)_n$ converges $\mathbb{P}$-almost everywhere to a non-negative invariant function $g$ such that
\begin{equation}\label{Kingman}
\int g\, d\mathbb{P} =\lim_{n\to\infty}  \frac{1}{n} 
\int u_n\, d\mathbb{P} .
\end{equation}

Now consider the sequence of measurable functions $h_{n}\colon \Omega\to [0,+\infty)$ defined by
 $$
 h_{n}(\omega)\eqdef \mbox{diam}\,f_{\theta^{-n}(\omega)}^{n}(X)=\mbox{diam}\,f_{\theta^{-1}(\omega)}\circ \dots \circ f_{\theta^{-n}(\omega)}(X)
 $$ 
 By hypothesis, we have that  $\lim_{n\to \infty} h_{n}(\omega)=0$ for $\mathbb{P}$-almost every $\omega$. In particular, 
 it follows from the dominated convergence theorem
 that
\begin{equation}\label{hypooo}
\lim_{n\to+\infty}\frac{1}{n}\sum_{j=1}^n \int h_j \, d\mathbb{P}=0,
\end{equation}
for $\mathbb{P}$-almost every $\omega$. Also, noting that
$
g_{n}\circ \theta^{-n}=h_{n}
$,
and using the invariance of $\theta$ we have 
\begin{equation}\label{igual}
 \int h_i \, d\mathbb{P} = \int g_i \, d\mathbb{P}
\end{equation}
for every $i\geq 1$. Therefore, from \eqref{Kingman} and \eqref{hypooo} we get
\[
\begin{split}
		\int g \, d\mathbb{P} &= \lim_{n\to+\infty} \frac{1}{n} \int u_n \, d\mathbb{P}\\
		&= \lim_{n\to+\infty}\frac{1}{n} \sum_{j=1}^n \int g_j \, d\mathbb{P}
		=
		 \lim_{n\to+\infty}\frac{1}{n} \sum_{j=1}^n \int h_j \, d\mathbb{P}=0.
	\end{split}
	\]
Since $g(\omega)\geq 0$ for $\mathbb{P}$-almost every $\omega$, we conclude that 
$g(\omega)=0$ for $\mathbb{P}$-almost every $\omega$, which proves that (1)$\Rightarrow$(2).

	We now prove that (2) $\Rightarrow$ (1). First note that $h_{n+1}(\omega)\leq h_{n}(\omega)$ for every 
	$n\geq 1$ and $\omega\in\Omega$. Hence the limit
	$$
	h(\omega)\eqdef \lim_{n\to \infty} h_{n}(\omega)
	$$
	exists for every $\omega\in \Omega$. In particular, using the dominated convergence theorem and \eqref{igual}
	we have
	$$
	\int h(\omega)\,d\mathbb{P}=\lim_{n\to \infty}\int\frac{1}{n} \sum_{i=1}^n  h_i \,
		 d\mathbb{P}=\lim_{n\to \infty}\int \frac{1}{n} \sum_{i=1}^n  g_i\,
		 d\mathbb{P}=0,
	$$
where the last equality follows from the hypothesis  $\lim_{n\to \infty}\frac{1}{n}\sum_{i=1}^{n} g_{i}(\omega)=0$
for $\mathbb{P}$-almost every $\omega$.
Since $h(\omega)\geq 0$ for every $\omega$, we conclude that $h(\omega)=0$ for $\mathbb{P}$-almost every $\omega$, which completes 
the proof of (2)$\Rightarrow$  (1).	
          
\hfill \qed

\bibliographystyle{acm}
\bibliography{references}

\end{document}